\documentclass{amsart}
\usepackage{amsfonts}

\newtheorem{theorem}{Theorem}
\theoremstyle{plain}

\newtheorem{definition}{Definition}

\newtheorem{lemma}{Lemma}

\numberwithin{equation}{section}
\input{tcilatex}

\begin{document}
\title[Isophote curves on timelike surfaces]{Isophote curves on timelike
surfaces in Minkowski 3-space}
\author{Fatih Do\u{g}an}
\address{\textit{Current Adress: Fatih} \textit{DO\u{G}AN}, \textit{Ankara
University, Department of Mathematics, 06100 Tando\u{g}an, Ankara, Turkey}}
\email{mathfdogan@hotmail.com}
\date{}

\begin{abstract}
Isophote comprises a locus of the surface points whose normal vectors make a
constant angle with a fixed vector. In this paper, isophote curves are
studied on timelike surfaces in Minkowski 3-space $E_{1}^{3}$. The axises of
spacelike and timelike isophote curves are found via their Darboux frames.
Subsequently, the relationship between isophotes and slant helices is shown
on timelike surfaces.
\end{abstract}

\maketitle

\section{Introduction}

Isophote is one of the characteristic curves on a surface such as parameter,
geodesic and asymptotic curves or lines of curvature.

Isophote on a surface can be regarded as a nice consequence of Lambert's
cosine law in optics branch of physics. Lambert's law states that the
intensity of illumination on a diffuse surface is proportional to the cosine
of the angle generated between the surface normal vector $N$ and the light
vector $d$. According to this law the intensity is irrespective of the
actual viewpoint, hence the illumination is the same when viewed from any
direction $[19]$. In other words, isophotes of a surface are curves with the
property that their points have the same light intensity from a given source
(a curve of constant illumination intensity). When the source light is at
infinity, we may consider that the light flow consists in parallel lines.
Hence, we can give a geometric description of isophotes on surfaces, namely
they are curves such that the surface normal vectors in points of the curve
make a constant angle with a fixed direction (which represents the light
direction). These curves are succesfully used in computer graphics but also
it is interesting to study for geometry.\newline
Then, to find an isophote on a surface we use the formula%
\begin{equation*}
\frac{\left \langle N(u,v),d\right \rangle }{\left \Vert N(u,v)\right \Vert }%
=\cos \theta ,\text{ }0\leq \theta \leq \frac{\pi }{2}.
\end{equation*}%
where $d$ is the light (fixed) vector and $\theta $ is the constant angle
between the surface normal vector $N$ and $d$.

Koenderink and van Doorn $[9]$ studied the field of constant image
brightness contours (isophotes). They showed that the spherical image (the
Gauss map) of an isophote is a latitude circle on the unit sphere $S^{2}$
and the problem was reduced to that of obtaining the inverse Gauss map of
these circles. By means of this they defined two kind singularities of the
Gauss map: folds (curves) and simple cusps (apex, antapex points) and there
are structural properties of the field of isophotes that bear an invariant
relation to geometric features of the object.

Poeschl $[16]$ used isophotes in car body construction via detecting
irregularities along these curves on a free form surface. These
irregularities emerge by differentiating of the equation $\left \langle
N(u,v),l\right \rangle =\cos \theta =c$ (constant)%
\begin{equation*}
\left \langle N_{u},l\right \rangle du+\left \langle N_{v},l\right \rangle
dv=0
\end{equation*}%
\begin{equation*}
\frac{dv}{du}=-\frac{\left \langle N_{u},l\right \rangle }{\left \langle
N_{v},l\right \rangle },\text{ }\left \langle N_{v},l\right \rangle \neq 0,
\end{equation*}%
where $l$ ($d$) is the light vector.

Sara $[18]$ researched local shading of a surface through isophotes
properties. By using fundamental theory of surfaces, he focused on accurate
estimation of surface normal tilt and on qualitatively correct Gaussian
curvature recovery.

Kim and Lee $[8]$ parameterized isophotes for surface of rotation and canal
surface. They utilized that both these surfaces decompose into a set of
circles where the surface normal vectors at points on each circle construct
a cone. Again the vectors that make a constant angle with the fixed vector $%
d $ construct another cone and thus tangential intersection of these cones
give the parametric range of the connected component isophote.

Dillen $et$ $al$. $[2]$ studied the constant angle surfaces in the product
space $\mathbb{S}^{2}\times 
%TCIMACRO{\U{211d} }%
%BeginExpansion
\mathbb{R}
%EndExpansion
$ for which the unit normal makes a constant angle with the $%
%TCIMACRO{\U{211d} }%
%BeginExpansion
\mathbb{R}
%EndExpansion
$-direction. Then Dillen and Munteanu $[3]$ investigated the same problem in 
$\mathbb{H}^{2}\times 
%TCIMACRO{\U{211d} }%
%BeginExpansion
\mathbb{R}
%EndExpansion
$ where $\mathbb{H}^{2}$ is the hyperbolic plane. Again, Nistor $[13]$
researched normal, binormal and tangent developable surfaces of the space
curve from standpoint of constant angle surface. Recently, Munteanu and
Nistor $[12]$ gave an important characterization about constant angle
surfaces and studied the constant angle surfaces taking with a fixed vector
direction being the tangent direction to $%
%TCIMACRO{\U{211d} }%
%BeginExpansion
\mathbb{R}
%EndExpansion
$ in Euclidean $3$-space. Thus, it can be said that all curves on a constant
angle surface are isophote curves.

Izumiya and Takeuchi $[6]$ defined a slant helix as a space curve that the
principal normal lines make a constant angle with a fixed direction. They
displayed that a certain slant helix is also a geodesic on the tangent
developable surface of a general helix.

Recently, Ali and Lopez $[1]$ looked into slant helices in Lorentz-Minkowski
space $E_{1}^{3}$. They gave characterizations as to slant helix and its
axis in $E_{1}^{3}$.

More recently, Do\u{g}an and Yayl\i \ $[4]$ have investigated isophote
curves in the Euclidean space $E^{3}$. Also, they $[5]$ studied isophote
curves on spacelike surfaces in $E_{1}^{3}$. In both papers they viewed that
the close relation between isophote curves and special curves on the
surfaces. For instance, an isophote can be generated by a curve which is
both geodesic and slant helix.

This time we study isophote curves on timelike surfaces in $E_{1}^{3}$. The
present paper is organized as follows. We give basic concepts concerning
curve and surface theory in section 2. In section 3 and 4, we focus on
finding the axises of spacelike and timelike isophote curves lying on
timelike surfaces. Finally, we give main theorems for these curves in
section 5.

\section{Preliminaries}

First of all, we begin to introduce Minkowski 3-space. Later, we mention
some fundamental concepts of curves and surfaces in the Minkowski 3-space $%
E_{1}^{3}$. The space $R_{1}^{3}$ is a three dimensional real vector space
endowed with the inner product%
\begin{equation*}
\left \langle x,y\right \rangle =-x_{1}y_{1}+x_{2}y_{2}+x_{3}y_{3}.
\end{equation*}%
This space is called Minkowski 3-space and denoted by $E_{1}^{3}$. A vector
in this space is said to be spacelike, timelike and lightlike (null) if $%
\left \langle x,x\right \rangle >0$ or $x=0$, $\left \langle
x,x\right
\rangle <0 $ and $\left \langle x,x\right \rangle =0$ or $x\neq 0$%
, respectively. Again, a regular curve $\alpha :I\longrightarrow $ $%
E_{1}^{3} $ is called spacelike, timelike and lightlike if the velocity
vector $\alpha ^{^{\prime }}$ is spacelike, timelike and lightlike,
respectively $[10]$.\newline
The Lorentzian cross product of $x=(x_{1},x_{2},x_{3})$ and $%
y=(y_{1},y_{2},y_{3})$ in $R_{1}^{3}$ is defined as follows.%
\begin{equation*}
x\times y=%
\begin{vmatrix}
e_{1} & -e_{2} & -e_{3} \\ 
x_{1} & x_{2} & x_{3} \\ 
y_{1} & y_{2} & y_{3}%
\end{vmatrix}%
=\left(
x_{2}y_{3}-x_{3}y_{2},x_{1}y_{3}-x_{3}y_{1},x_{2}y_{1}-x_{1}y_{2}\right) ,
\end{equation*}%
where $\delta _{ij}$ is kronecker delta, $e_{i}=(\delta _{i1},\delta
_{i2},\delta _{i3})$ and $e_{1}\times e_{2}=-e_{3}$, $e_{2}\times
e_{3}=e_{1} $, $e_{3}\times e_{1}=-e_{2}$.

Let $\{t,n,b\}$ be the moving Frenet frame along the curve $\alpha $ with
arclenght parameter $s$. For a spacelike curve $\alpha $, the Frenet-Serret
equations are%
\begin{equation*}
\begin{bmatrix}
t^{^{\prime }} \\ 
n^{^{\prime }} \\ 
b^{^{\prime }}%
\end{bmatrix}%
=%
\begin{bmatrix}
0 & \kappa & 0 \\ 
-\varepsilon \kappa & 0 & \tau \\ 
0 & \tau & 0%
\end{bmatrix}%
\begin{bmatrix}
t \\ 
n \\ 
b%
\end{bmatrix}%
,
\end{equation*}%
where $\left \langle t,t\right \rangle =1$, $\left \langle n,n\right \rangle
=\pm 1$, $\left \langle b,b\right \rangle =-\varepsilon $, $\left \langle
t,n\right \rangle =\left \langle t,b\right \rangle =\left \langle
n,b\right
\rangle =0$ and $\kappa $ is the curvature and $\tau $ is the
torsion of $\alpha $. Here, $\varepsilon $ determines the kind of spacelike
curve $\alpha $. If $\varepsilon =1$, then $\alpha (s)$ is a spacelike curve
with spacelike principal normal $n$ and timelike binormal $b$. If $%
\varepsilon =-1 $, then $\alpha (s)$ is a spacelike curve with timelike
principal normal $n$ and spacelike binormal $b$.

If the curve $\alpha $ is a timelike, the Frenet-Serret equations are%
\begin{equation*}
\begin{bmatrix}
t^{^{\prime }} \\ 
n^{^{\prime }} \\ 
b^{^{\prime }}%
\end{bmatrix}%
=%
\begin{bmatrix}
0 & \kappa & 0 \\ 
\kappa & 0 & \tau \\ 
0 & -\tau & 0%
\end{bmatrix}%
\begin{bmatrix}
t \\ 
n \\ 
b%
\end{bmatrix}%
\end{equation*}%
where $\left \langle t,t\right \rangle =-1$, $\left \langle
n,n\right
\rangle =\left \langle b,b\right \rangle =1$, $\left \langle
t,n\right
\rangle =\left \langle t,b\right \rangle =\left \langle
n,b\right
\rangle =0$ $[20]$.

\begin{definition}[{[14]}]
Let $v$ and $w$ be spacelike vectors.\newline
\textbf{(a)} If $v$ and $w$ span a timelike vector subspace, then there is a
unique non-negative real number $\theta \geq 0$ such that%
\begin{equation}
\left \langle v,\omega \right \rangle =\left \Vert v\right \Vert \left \Vert
w\right \Vert \cosh \theta .  \tag{2.1}
\end{equation}%
\textbf{(b)} If $v$ and $w$ span a spacelike vector subspace, then there is
a unique non-negative real number $\theta \geq 0$ such that%
\begin{equation}
\left \langle v,\omega \right \rangle =\left \Vert v\right \Vert \left \Vert
w\right \Vert \cos \theta .  \tag{2.2}
\end{equation}
\end{definition}

\begin{definition}[{[14]}]
Let $v$ be a spacelike vector and $w$ be a timelike vector in $R_{1}^{3}$.
Then, there is a unique non-negative real number $\theta \geq 0$ such that%
\begin{equation}
\left \langle v,w\right \rangle =\left \Vert v\right \Vert \left \Vert
w\right \Vert \sinh \theta .  \tag{2.3}
\end{equation}
\end{definition}

\begin{definition}[{[14]}]
Let $v$ and $w$ be in the same timecone of $R_{1}^{3}$. Then, there is a
unique real number $\theta \geq 0$, called the hyperbolic angle between $v$
and $w$ such that%
\begin{equation}
\left \langle v,\omega \right \rangle =-\left \Vert v\right \Vert \left
\Vert w\right \Vert \cosh \theta .  \tag{2.4}
\end{equation}
\end{definition}

\begin{lemma}
In the Minkowski 3-space $E_{1}^{3}$, we have the following $[20]$.\newline
(i) Two timelike vectors cannot be orthogonal.\newline
(ii) Two null vectors are orthogonal if and only if they are linearly
dependent.\newline
(iii) A timelike vector cannot be orthogonal to a null (lightlike) vector.
\end{lemma}

Let $M$ be a regular timelike surface in $E_{1}^{3}$ and let $\alpha
:I\subset 
%TCIMACRO{\U{211d} }%
%BeginExpansion
\mathbb{R}
%EndExpansion
\longrightarrow M$ be a unit speed spacelike curve. Then, Darboux frame $%
\{T, $ $B=N\times T,$ $N\}$ is well-defined and positively oriented along
the curve $\alpha $ where $T$ is the tangent of $\alpha $ and $N$ is the
unit normal of $M$. In this case, the Darboux equations are given by%
\begin{gather}
T^{^{\prime }}=k_{g}B-k_{n}N  \tag{2.5} \\
B^{^{\prime }}=k_{g}T+\tau _{g}N  \notag \\
N^{^{\prime }}=k_{n}T+\tau _{g}B,  \notag
\end{gather}%
where $k_{n}$, $k_{g}$ and $\tau _{g}$ are the normal curvature, the
geodesic curvature and the geodesic torsion of $\alpha $, respectively and $%
\left \langle T,T\right \rangle =\left \langle N,N\right \rangle
=\left
\langle n,n\right \rangle =1$, $\left \langle B,B\right \rangle =-1$%
. Then, by using Eq.(2.5) we get%
\begin{gather}
\kappa ^{2}=k_{n}^{2}-k_{g}^{2}  \tag{2.6} \\
k_{g}=\kappa \sinh \phi  \notag \\
k_{n}=\kappa \cosh \phi  \notag \\
\tau _{g}=\tau +\phi ^{^{\prime }},  \notag
\end{gather}%
where $\phi $ is the angle between the surface normal vector $N$ and the
principal normal $n$ of $\alpha $.

If $\alpha :I\subset 
%TCIMACRO{\U{211d} }%
%BeginExpansion
\mathbb{R}
%EndExpansion
\longrightarrow M$ is a timelike curve, then the Darboux equations are given
by%
\begin{gather}
T^{^{\prime }}=k_{g}B+k_{n}N  \tag{2.7} \\
B^{^{\prime }}=k_{g}T-\tau _{g}N  \notag \\
N^{^{\prime }}=k_{n}T+\tau _{g}B,  \notag
\end{gather}%
where $\left \langle T,T\right \rangle =-1$, $\left \langle
N,N\right
\rangle =\left \langle B,B\right \rangle =\left \langle
n,n\right
\rangle =1$. From Eq.(2.7) we get,%
\begin{gather}
\kappa ^{2}=k_{g}^{2}+k_{n}^{2}  \tag{2.8} \\
k_{g}=\kappa \cos \phi  \notag \\
k_{n}=\kappa \sin \phi  \notag \\
\tau _{g}=\tau +\phi ^{^{\prime }},  \notag
\end{gather}%
where $\phi $ is the angle between the surface normal vector $N$ and the
principal normal $n$ of $\alpha $.

\section{The axis of a spacelike isophote curve on timelike surfaces}

Now, we find the fixed vector (axis) of a spacelike isophote curve via its
Darboux frame. Let $M$ be a timelike surface and let $\alpha $ be a unit
speed spacelike isophote curve on $M$. Then, there are two cases for the
axis $d$ of $\alpha $.\newline
\textbf{The Case (1).} If the axis $d$ is spacelike vector, then from
Definition 1(a) and 1(b) we have%
\begin{equation*}
\left \langle N,d\right \rangle =\cosh \theta \text{ \  \  \  \ or \  \  \  \ }%
\left \langle N,d\right \rangle =\cos \beta .
\end{equation*}%
where $\theta $ and $\beta $ are the constant angles between the surface
normal vector $N$ and $d$, respectively.\newline
\textbf{(a)} Let\textbf{\ }$\left \langle N,d\right \rangle =\cosh \theta $.
If we differentiate this equation with respect to$\ s$ along the curve $%
\alpha $, by Eq.(2.5) we get%
\begin{gather*}
\left \langle N^{^{\prime }},d\right \rangle =0 \\
\left \langle k_{n}T+\tau _{g}B,d\right \rangle =0 \\
k_{n}\left \langle T,d\right \rangle +\tau _{g}\left \langle B,d\right
\rangle =0 \\
\left \langle T,d\right \rangle =-\frac{\tau _{g}}{k_{n}}\left \langle
B,d\right \rangle
\end{gather*}%
If we take $\left \langle B,d\right \rangle =a$, the axis $d$ can be written
as%
\begin{equation*}
d=-\frac{\tau _{g}}{k_{n}}aT-aB+\cosh \theta N,
\end{equation*}%
where $\left \langle T,T\right \rangle =\left \langle N,N\right \rangle =1$
and $\left \langle B,B\right \rangle =-1$. Since $d$ is spacelike, we obtain%
\begin{gather*}
\left \langle d,d\right \rangle =\frac{\tau _{g}^{2}}{k_{n}^{2}}%
a^{2}-a^{2}+\cosh ^{2}\theta =1 \\
(1-\frac{\tau _{g}^{2}}{k_{n}^{2}})a^{2}=\sinh ^{2}\theta \\
a=\mp \frac{k_{n}}{\sqrt{k_{n}^{2}-\tau _{g}^{2}}}\sinh \theta .
\end{gather*}%
By substituting this in the expression of $d$, we get the axis as%
\begin{equation}
d=\pm \frac{\tau _{g}}{\sqrt{k_{n}^{2}-\tau _{g}^{2}}}\sinh \theta T\pm 
\frac{k_{n}}{\sqrt{k_{n}^{2}-\tau _{g}^{2}}}\sinh \theta B+\cosh \theta N. 
\tag{3.1}
\end{equation}%
If we differentiate $N^{^{\prime }}$ with respect to$\ s$ and take inner
product with $d$, we conclude%
\begin{equation}
N^{^{\prime \prime }}=(k_{n}^{^{\prime }}+k_{g}\tau _{g})T+(\tau
_{g}^{^{\prime }}+k_{n}k_{g})B-(k_{n}^{2}-\tau _{g}^{2})N  \tag{3.2}
\end{equation}%
\begin{equation*}
\left \langle N^{^{\prime \prime }},d\right \rangle =\mp \frac{(\tau
_{g}^{^{\prime }}k_{n}-k_{n}^{^{\prime }}\tau _{g})+k_{g}(k_{n}^{2}-\tau
_{g}^{2})}{\sqrt{k_{n}^{2}-\tau _{g}^{2}}}\sinh \theta -(k_{n}^{2}-\tau
_{g}^{2})\cosh \theta =0
\end{equation*}%
\begin{gather}
\coth \theta =\mp \left[ \frac{\tau _{g}^{^{\prime }}k_{n}-k_{n}^{^{\prime
}}\tau _{g}}{(k_{n}^{2}-\tau _{g}^{2})^{\frac{3}{2}}}+\frac{k_{g}}{%
(k_{n}^{2}-\tau _{g}^{2})^{\frac{1}{2}}}\right]  \notag \\
=\mp \left[ \frac{k_{n}^{2}}{(k_{n}^{2}-\tau _{g}^{2})^{\frac{3}{2}}}\left( 
\frac{\tau _{g}}{k_{n}}\right) ^{^{\prime }}+\frac{k_{g}}{(k_{n}^{2}-\tau
_{g}^{2})^{\frac{1}{2}}}\right] .  \tag{3.3}
\end{gather}%
Indeed, $d$ is a constant vector. If we differentiate the vector $d$, from
Eq.(2.5) we get%
\begin{eqnarray*}
d^{^{\prime }} &=&\pm \sinh \theta \left[ (\frac{\tau _{g}}{\sqrt{%
k_{n}^{2}-\tau _{g}^{2}}})^{^{\prime }}T+\frac{\tau _{g}}{\sqrt{%
k_{n}^{2}-\tau _{g}^{2}}}(k_{g}B-k_{n}N)\right] \\
&&\pm \sinh \theta \left[ (\frac{k_{n}}{\sqrt{k_{n}^{2}-\tau _{g}^{2}}}%
)^{^{\prime }}B+\frac{k_{n}}{\sqrt{k_{n}^{2}-\tau _{g}^{2}}}(k_{g}T+\tau
_{g}B)\right] +\cosh \theta \left[ k_{n}T+\tau _{g}B\right]
\end{eqnarray*}%
By Eq.(3.3), we have%
\begin{equation*}
\cosh \theta =\pm \sinh \theta \left[ \frac{k_{n}^{^{\prime }}\tau _{g}-\tau
_{g}^{^{\prime }}k_{n}}{(k_{n}^{2}-\tau _{g}^{2})^{\frac{3}{2}}}-\frac{k_{g}%
}{(k_{n}^{2}-\tau _{g}^{2})^{\frac{1}{2}}}\right] .
\end{equation*}%
The last equality is replaced in the statement of $d^{^{\prime }}$, it
follows that%
\begin{eqnarray*}
d^{^{\prime }} &=&\pm \sinh \theta \left[ \frac{\tau _{g}^{^{\prime
}}(k_{n}^{2}-\tau _{g}^{2})-\tau _{g}(k_{n}k_{n}^{^{\prime }}-\tau _{g}\tau
_{g}^{^{\prime }})}{(k_{n}^{2}-\tau _{g}^{2})^{\frac{3}{2}}}+\frac{%
k_{n}k_{n}^{^{\prime }}\tau _{g}-k_{n}^{2}\tau _{g}^{^{\prime }}}{%
(k_{n}^{2}-\tau _{g}^{2})^{\frac{3}{2}}}\right] T \\
&&\pm \sinh \theta \left[ \frac{k_{n}^{^{\prime }}(k_{n}^{2}-\tau
_{g}^{2})-k_{n}(k_{n}k_{n}^{^{\prime }}-\tau _{g}\tau _{g}^{^{\prime }})}{%
(k_{n}^{2}-\tau _{g}^{2})^{\frac{3}{2}}}+\frac{k_{n}^{^{\prime }}\tau
_{g}^{2}-k_{n}\tau _{g}\tau _{g}^{^{\prime }}}{(k_{n}^{2}-\tau _{g}^{2})^{%
\frac{3}{2}}}\right] B.
\end{eqnarray*}%
As can be immediately seen above, the coefficients of $T$ and $B$ becomes
zero. Then $d^{^{\prime }}=0$ in other words $d$ is a constant vector.%
\newline
\textbf{(b)} Let $\left \langle N,d\right \rangle =\cos \beta $. In that
case, by Eq.(2.5) it concludes that%
\begin{equation*}
\left \langle T,d\right \rangle =-\frac{\tau _{g}}{k_{n}}\left \langle
B,d\right \rangle .
\end{equation*}%
If we take $\left \langle B,d\right \rangle =a$, the axis $d$ can be written
as%
\begin{equation*}
d=-\frac{\tau _{g}}{k_{n}}aT-aB+\cos \beta N.
\end{equation*}%
where $\left \langle T,T\right \rangle =\left \langle N,N\right \rangle =1$
and $\left \langle B,B\right \rangle =-1$. Since $d$ is spacelike, we obtain%
\begin{eqnarray*}
\left \langle d,d\right \rangle &=&\frac{\tau _{g}^{2}}{k_{n}^{2}}%
a^{2}-a^{2}+\cos ^{2}\beta =1 \\
a &=&\mp \frac{k_{n}}{\sqrt{\tau _{g}^{2}-k_{n}^{2}}}\sin \beta .
\end{eqnarray*}%
In this case, the axis $d$ becomes%
\begin{equation}
d=\pm \frac{\tau _{g}}{\sqrt{\tau _{g}^{2}-k_{n}^{2}}}\sin \beta T\pm \frac{%
k_{n}}{\sqrt{\tau _{g}^{2}-k_{n}^{2}}}\sin \beta B+\cos \beta N.  \tag{3.4}
\end{equation}%
From Eq.(3.2) we have,%
\begin{equation*}
N^{^{\prime \prime }}=(k_{n}^{^{\prime }}+k_{g}\tau _{g})T+(\tau
_{g}^{^{\prime }}+k_{n}k_{g})B-(k_{n}^{2}-\tau _{g}^{2})N.
\end{equation*}%
By taking inner product of $N^{^{\prime \prime }}$and $d$, we get%
\begin{equation*}
\left \langle N^{^{\prime \prime }},d\right \rangle =\mp \frac{(\tau
_{g}^{^{\prime }}k_{n}-k_{n}^{^{\prime }}\tau _{g})-k_{g}(\tau
_{g}^{2}-k_{n}^{2})}{\sqrt{\tau _{g}^{2}-k_{n}^{2}}}\sin \beta +(\tau
_{g}^{2}-k_{n}^{2})\cos \beta =0
\end{equation*}%
\begin{equation}
\cot \beta =\pm \left[ \frac{k_{n}^{2}}{(\tau _{g}^{2}-k_{n}^{2})^{\frac{3}{2%
}}}\left( \frac{\tau _{g}}{k_{n}}\right) ^{^{\prime }}-\frac{k_{g}}{(\tau
_{g}^{2}-k_{n}^{2})^{\frac{1}{2}}}\right] .  \tag{3.5}
\end{equation}%
If we differentiate Eq.(3.4) and then use Eq.(3.5), we obtain%
\begin{eqnarray*}
d^{^{\prime }} &=&\pm \sin \beta \left[ \frac{\tau _{g}^{^{\prime }}(\tau
_{g}^{2}-k_{n}^{2})-\tau _{g}(\tau _{g}\tau _{g}^{^{\prime
}}-k_{n}k_{n}^{^{\prime }})}{(\tau _{g}^{2}-k_{n}^{2})^{\frac{3}{2}}}+\frac{%
k_{n}^{2}\tau _{g}^{^{\prime }}-k_{n}k_{n}^{^{\prime }}\tau _{g}}{(\tau
_{g}^{2}-k_{n}^{2})^{\frac{3}{2}}}\right] T \\
&&\pm \sin \beta \left[ \frac{k_{n}^{^{\prime }}(\tau
_{g}^{2}-k_{n}^{2})-k_{n}(\tau _{g}\tau _{g}^{^{\prime
}}-k_{n}k_{n}^{^{\prime }})}{(k_{n}^{2}-\tau _{g}^{2})^{\frac{3}{2}}}+\frac{%
k_{n}\tau _{g}\tau _{g}^{^{\prime }}-k_{n}^{^{\prime }}\tau _{g}^{2}}{%
(k_{n}^{2}-\tau _{g}^{2})^{\frac{3}{2}}}\right] B.
\end{eqnarray*}%
Since the coefficients of $T$ and $B$ are zero, $d^{^{\prime }}=0$, i.e., $d$
is a constant vector.\newline
\textbf{The Case (2)} If the axis $d$ is a timelike vector, then from
Definition 2 we have%
\begin{equation*}
\left \langle N,d\right \rangle =\sinh \gamma ,
\end{equation*}%
where $\gamma $ is the constant angle between the surface normal vector $N$
and $d$. By doing computations similar to the case (1) we get%
\begin{gather}
d=\pm \frac{\tau _{g}}{\sqrt{k_{n}^{2}-\tau _{g}^{2}}}\cosh \gamma T\pm 
\frac{k_{n}}{\sqrt{k_{n}^{2}-\tau _{g}^{2}}}\cosh \gamma B+\sinh \gamma N 
\notag \\
\tanh \gamma =\mp \left[ \frac{k_{n}^{2}}{(k_{n}^{2}-\tau _{g}^{2})^{\frac{3%
}{2}}}\left( \frac{\tau _{g}}{k_{n}}\right) ^{^{\prime }}+\frac{k_{g}}{%
(k_{n}^{2}-\tau _{g}^{2})^{\frac{1}{2}}}\right]  \tag{3.6}
\end{gather}%
and again similar to proof of the case (1) it can be showed that $d$ is a
constant vector.

From now on, we will obtain the axis of timelike isophote curves on timelike
surfaces.

\section{The axis of a timelike isophote curve on timelike surfaces}

In this section, we find the fixed vector (axis) of a timelike isophote
curve via its Darboux frame. Let $M$ be a timelike surface and let $\alpha $
be a unit speed timelike isophote curve on $M$. Then, there are two cases
for the axis $d$ of $\alpha $.\newline
\textbf{The Case (3).} If the axis $d$ is spacelike, then from Definition
1(b) and 1(a) we have%
\begin{equation*}
\left \langle N,d\right \rangle =\cos \delta \text{ \  \  \  \ or \  \  \  \ }%
\left \langle N,d\right \rangle =\cosh \xi .
\end{equation*}%
where $\delta $ and $\xi $ are the constant angles between the surface
normal vector $N$ and $d$, respectively.\newline
\textbf{(a) }Let $\left \langle N,d\right \rangle =\cos \delta $. Then, from
Eq.(2.7) it follows that%
\begin{equation*}
\left \langle T,d\right \rangle =-\frac{\tau _{g}}{k_{n}}\left \langle
B,d\right \rangle .
\end{equation*}%
By taking $\left \langle B,d\right \rangle =a$, the axis $d$ can be written
as%
\begin{equation*}
d=\frac{\tau _{g}}{k_{n}}aT+aB+\cos \delta N,
\end{equation*}%
where $\left \langle T,T\right \rangle =-1$ and $\left \langle
N,N\right
\rangle =\left \langle B,B\right \rangle =1$. Since $d$ is
spacelike, we obtain%
\begin{equation*}
a=\pm \frac{k_{n}}{\sqrt{k_{n}^{2}-\tau _{g}^{2}}}\sin \delta .
\end{equation*}%
Then we have%
\begin{gather}
d=\pm \frac{\tau _{g}}{\sqrt{k_{n}^{2}-\tau _{g}^{2}}}\sin \delta T\pm \frac{%
k_{n}}{\sqrt{k_{n}^{2}-\tau _{g}^{2}}}\sin \delta B+\cos \delta N  \notag \\
\cot \delta =\mp \left[ \frac{k_{n}^{2}}{(k_{n}^{2}-\tau _{g}^{2})^{\frac{3}{%
2}}}\left( \frac{\tau _{g}}{k_{n}}\right) ^{^{\prime }}+\frac{k_{g}}{%
(k_{n}^{2}-\tau _{g}^{2})^{\frac{1}{2}}}\right] .  \tag{4.1}
\end{gather}%
Like preceding cases it can be showed that $d$ is a constant vector.\newline
\textbf{(b)} Let $\left \langle N,d\right \rangle =\cosh \xi $. Then we can
easily obtain%
\begin{gather}
d=\pm \frac{\tau _{g}}{\sqrt{\tau _{g}^{2}-k_{n}^{2}}}\sinh \xi T\pm \frac{%
k_{n}}{\sqrt{\tau _{g}^{2}-k_{n}^{2}}}\sinh \xi B+\cosh \xi N  \notag \\
\coth \xi =\pm \left[ \frac{k_{n}^{2}}{(\tau _{g}^{2}-k_{n}^{2})^{\frac{3}{2}%
}}\left( \frac{\tau _{g}}{k_{n}}\right) ^{^{\prime }}+\frac{k_{g}}{(\tau
_{g}^{2}-k_{n}^{2})^{\frac{1}{2}}}\right] .  \tag{4.2}
\end{gather}%
\textbf{The Case (4). }If the axis $d$ is a timelike vector, then from
Definition 2 we have%
\begin{equation*}
\left \langle N,d\right \rangle =\sinh \nu .
\end{equation*}%
where $\nu $ is the constant angle between the surface normal vector $N$ and 
$d$. In this situation, we get%
\begin{gather}
d=\pm \frac{\tau _{g}}{\sqrt{\tau _{g}^{2}-k_{n}^{2}}}\cosh \nu T\pm \frac{%
k_{n}}{\sqrt{\tau _{g}^{2}-k_{n}^{2}}}\cosh \nu B+\sinh \nu N  \notag \\
\coth \nu =\pm \left[ \frac{k_{n}^{2}}{(\tau _{g}^{2}-k_{n}^{2})^{\frac{3}{2}%
}}\left( \frac{\tau _{g}}{k_{n}}\right) ^{^{\prime }}-\frac{k_{g}}{(\tau
_{g}^{2}-k_{n}^{2})^{\frac{1}{2}}}\right] .  \tag{4.3}
\end{gather}

\section{Main theorems}

In this section, we give main theorems that characterize isophotes on
timelike surfaces. Moreover, we show the relationship between isophote
curves and slant helices on timelike surfaces.

\begin{theorem}
A unit speed spacelike curve on a timelike surface is an isophote curve if
and only if one of the following three functions%
\begin{eqnarray*}
(1)\text{ \ }\coth \theta &=&\eta (s)=\mp \left( \frac{k_{n}^{2}}{%
(k_{n}^{2}-\tau _{g}^{2})^{\frac{3}{2}}}\left( \frac{\tau _{g}}{k_{n}}%
\right) ^{^{\prime }}+\frac{k_{g}}{(k_{n}^{2}-\tau _{g}^{2})^{\frac{1}{2}}}%
\right) (s) \\
(2)\text{ \  \ }\cot \beta &=&\mu (s)=\pm \left( \frac{k_{n}^{2}}{(\tau
_{g}^{2}-k_{n}^{2})^{\frac{3}{2}}}\left( \frac{\tau _{g}}{k_{n}}\right)
^{^{\prime }}-\frac{k_{g}}{(\tau _{g}^{2}-k_{n}^{2})^{\frac{1}{2}}}\right)
(s) \\
(3)\text{ }\tanh \gamma &=&\psi (s)=\mp \left( \frac{k_{n}^{2}}{%
(k_{n}^{2}-\tau _{g}^{2})^{\frac{3}{2}}}\left( \frac{\tau _{g}}{k_{n}}%
\right) ^{^{\prime }}+\frac{k_{g}}{(k_{n}^{2}-\tau _{g}^{2})^{\frac{1}{2}}}%
\right) (s)
\end{eqnarray*}%
is a constant function (The case 1(a), the case 1(b) and the case 2,
respectively).
\end{theorem}

\begin{proof}
(1) Since $\alpha $ is an isophote, the Gauss map along the curve $\alpha $
is a circle on the Lorentzian unit sphere $S_{1}^{2}$. Hence, if we compute
the Gauss map $N_{\mid _{\alpha }}:I\longrightarrow $ $S_{1}^{2}$ along the
curve $\alpha $, the geodesic curvature of $N_{\mid _{\alpha }}$ becomes $%
\eta (s)$ as shown below.
\end{proof}

\begin{eqnarray*}
N_{\mid _{\alpha }}^{^{\prime }} &=&k_{n}T+\tau _{g}B \\
N_{\mid _{\alpha }}^{^{\prime \prime }} &=&(k_{n}^{^{\prime }}+k_{g}\tau
_{g})T+(k_{n}k_{g}+\tau _{g}^{^{\prime }})B-(k_{n}^{2}-\tau _{g}^{2})N \\
N_{\mid _{\alpha }}^{^{\prime }}\times N_{\mid _{\alpha }}^{^{\prime \prime
}} &=&-\tau _{g}(k_{n}^{2}-\tau _{g}^{2})T+k_{n}(k_{n}^{2}-\tau
_{g}^{2})B+(k_{g}(k_{n}^{2}-\tau _{g}^{2})+k_{n}^{2}(\frac{\tau _{g}}{k_{n}}%
)^{^{\prime }})N,
\end{eqnarray*}%
where $T\times B=N$, $B\times N=T$ and $N\times T=B$. Therefore, we obtain%
\begin{eqnarray*}
\kappa &=&\frac{\sqrt{\left \langle N_{\mid _{\alpha }}^{^{\prime }}\times
N_{\mid _{\alpha }}^{^{\prime \prime }},N_{\mid _{\alpha }}^{^{\prime
}}\times N_{\mid _{\alpha }}^{^{\prime \prime }}\right \rangle }}{\left
\Vert N_{\mid _{\alpha }}^{^{\prime }}\right \Vert ^{3}} \\
&=&\frac{\sqrt{-(k_{n}^{2}-\tau _{g}^{2})^{3}+\left( k_{g}(k_{n}^{2}-\tau
_{g}^{2})+k_{n}^{2}(\dfrac{\tau _{g}}{k_{n}})^{^{\prime }}\right) ^{2}}}{%
\sqrt{(k_{n}^{2}-\tau _{g}^{2})^{3}}} \\
&=&\sqrt{-1+\frac{\left( k_{g}(k_{n}^{2}-\tau _{g}^{2})+k_{n}^{2}(\dfrac{%
\tau _{g}}{k_{n}})^{^{\prime }}\right) ^{2}}{(k_{n}^{2}-\tau _{g}^{2})^{3}}}.
\end{eqnarray*}

Let $\overset{-}{k_{g}}$ and $\overset{-}{k_{n}}$ be the geodesic curvature
and the normal curvature of the Gauss map $N_{\mid _{\alpha }}$ on $%
S_{1}^{2} $, respectively. Since the normal curvature $\overset{-}{k_{n}}$ $%
=1$, if we substitute $\overset{-}{k_{n}}$ and $\kappa $ in the following
equation, we obtain the geodesic curvature $\overset{-}{k_{g}}$ as follows.%
\begin{equation*}
\kappa ^{2}=(\overset{-}{k_{g}})^{2}-(\overset{-}{k_{n}})^{2}
\end{equation*}%
\begin{equation*}
\overset{-}{k_{g}}(s)=\eta (s)=\coth \theta =\mp \left( \frac{k_{n}^{2}}{%
(k_{n}^{2}-\tau _{g}^{2})^{\frac{3}{2}}}\left( \frac{\tau _{g}}{k_{n}}%
\right) ^{^{\prime }}+\frac{k_{g}}{(k_{n}^{2}-\tau _{g}^{2})^{\frac{1}{2}}}%
\right) (s).
\end{equation*}%
where $\theta $ is the constant angle between the surface normal vector $N$
and $d$. In that case, the spherical images (Gauss maps) of isophotes are
circles if and only if one of three functions $\eta (s)$, $\mu (s)$ and $%
\psi (s)$ is a constant. The proofs for (2) and (3) can be done in the same
way.

\begin{theorem}
Let $\alpha $ be a unit speed spacelike curve in $E_{1}^{3}$. If the normal
vector of $\alpha $ is spacelike, then $\alpha $ is a slant helix if and
only if one of the two functions%
\begin{equation*}
\frac{\kappa ^{2}}{(\tau ^{2}-\kappa ^{2})^{\frac{3}{2}}}(\frac{\tau }{%
\kappa })^{^{\prime }}\text{ \  \  \  \ and \  \  \  \ }\frac{\kappa ^{2}}{(\kappa
^{2}-\tau ^{2})^{\frac{3}{2}}}(\frac{\tau }{\kappa })^{^{\prime }}
\end{equation*}%
is a constant everywhere $\tau ^{2}-\kappa ^{2}$ does not vanish $[1]$.
\end{theorem}

\begin{theorem}
Let $\alpha $ be a unit speed spacelike isophote curve on a timelike surface
(The case 1(a), the case 1(b) and the case 2, respectively). Then,\newline
\textbf{(a)} $\alpha $ is a geodesic on the timelike surface if and only if $%
\alpha $ is a slant helix with the spacelike axis%
\begin{equation*}
d=\pm \frac{\tau }{\sqrt{\kappa ^{2}-\tau ^{2}}}\sinh \theta T\pm \frac{%
\kappa }{\sqrt{\kappa ^{2}-\tau ^{2}}}\sinh \theta B+\cosh \theta N.
\end{equation*}%
\textbf{(b)} $\alpha $ is a geodesic on the timelike surface if and only if $%
\alpha $ is a slant helix with the spacelike axis%
\begin{equation*}
d=\pm \frac{\tau }{\sqrt{\tau ^{2}-\kappa ^{2}}}\sin \beta T\pm \frac{\kappa 
}{\sqrt{\tau ^{2}-\kappa ^{2}}}\sin \beta B+\cos \beta N.
\end{equation*}%
\textbf{(c)} $\alpha $ is a geodesic on the timelike surface if and only if $%
\alpha $ is a slant helix with the timelike axis%
\begin{equation*}
d=\pm \frac{\tau }{\sqrt{\kappa ^{2}-\tau ^{2}}}\cosh \gamma T\pm \frac{%
\kappa }{\sqrt{\kappa ^{2}-\tau ^{2}}}\cosh \gamma B+\sinh \gamma N.
\end{equation*}
\end{theorem}

\begin{proof}
\textit{(a)} Since $\alpha $ is a geodesic, we have $k_{g}=0$. From Eq.(2.6)
it follows that $k_{n}=\mp \kappa $ and $\tau _{g}=\tau $. By substituting $%
k_{g}$ and $k_{n}$ in the expression of $\eta (s)$ we get%
\begin{equation*}
\eta (s)=\mp \left( \frac{\kappa ^{2}}{(\kappa ^{2}-\tau ^{2})^{\frac{3}{2}}}%
(\frac{\tau }{\kappa })^{^{\prime }}\right) (s)
\end{equation*}%
is a constant function. Then, from Theorem 2 $\alpha $ is a slant helix.
Because $k_{n}=\mp \kappa $ and $\tau _{g}=\tau $, using Eq.(3.1) we obtain
the spacelike axis of slant helix as%
\begin{equation*}
d=\pm \frac{\tau }{\sqrt{\kappa ^{2}-\tau ^{2}}}\sinh \theta T\pm \frac{%
\kappa }{\sqrt{\kappa ^{2}-\tau ^{2}}}\sinh \theta B+\cosh \theta N.
\end{equation*}%
Conversely, let $\alpha $ be a slant helix with the spacelike axis%
\begin{equation*}
d=\pm \frac{\tau }{\sqrt{\kappa ^{2}-\tau ^{2}}}\sinh \theta T\pm \frac{%
\kappa }{\sqrt{\kappa ^{2}-\tau ^{2}}}\sinh \theta B+\cosh \theta N.
\end{equation*}%
Then from Eq.(3.1) we get $k_{n}=\mp \kappa $ and $\tau _{g}=\tau $. This
means that $k_{g}=0$, i.e., $\alpha $ is a geodesic on the timelike surface.%
\newline
The proof of \textit{(b)} and \textit{(c) can be done similar to the proof
of (a).}
\end{proof}

\begin{theorem}
Let $\alpha $ be a unit speed timelike curve in $E_{1}^{3}$. Then $\alpha $
is a slant helix if and only if one of the two functions%
\begin{equation*}
\frac{\kappa ^{2}}{(\tau ^{2}-\kappa ^{2})^{\frac{3}{2}}}(\frac{\tau }{%
\kappa })^{^{\prime }}\text{ \  \  \  \ and \  \  \  \ }\frac{\kappa ^{2}}{(\kappa
^{2}-\tau ^{2})^{\frac{3}{2}}}(\frac{\tau }{\kappa })^{^{\prime }}
\end{equation*}%
is a constant everywhere $\tau ^{2}-\kappa ^{2}$ does not vanish $[1]$.
\end{theorem}

\begin{theorem}
A unit speed timelike curve on a timelike surface is an isophote curve if
and only if one of the following three functions%
\begin{eqnarray*}
(1)\text{ \ }\cot \delta &=&\sigma (s)=\mp \left( \frac{k_{n}^{2}}{%
(k_{n}^{2}-\tau _{g}^{2})^{\frac{3}{2}}}\left( \frac{\tau _{g}}{k_{n}}%
\right) ^{^{\prime }}+\frac{k_{g}}{(k_{n}^{2}-\tau _{g}^{2})^{\frac{1}{2}}}%
\right) (s) \\
(2)\text{ }\coth \xi &=&\rho (s)=\pm \left( \frac{k_{n}^{2}}{(\tau
_{g}^{2}-k_{n}^{2})^{\frac{3}{2}}}\left( \frac{\tau _{g}}{k_{n}}\right)
^{^{\prime }}+\frac{k_{g}}{(\tau _{g}^{2}-k_{n}^{2})^{\frac{1}{2}}}\right)
(s) \\
(3)\text{ }\coth \nu &=&\omega (s)=\pm \left( \frac{k_{n}^{2}}{(\tau
_{g}^{2}-k_{n}^{2})^{\frac{3}{2}}}\left( \frac{\tau _{g}}{k_{n}}\right)
^{^{\prime }}-\frac{k_{g}}{(\tau _{g}^{2}-k_{n}^{2})^{\frac{1}{2}}}\right)
(s)
\end{eqnarray*}%
is a constant function (The case 3(a), the case 3(b) and the case 4,
respectively).\newline
The proof of Theorem 5 is similar to Theorem 1.
\end{theorem}

\begin{theorem}
Let $\alpha $ be a unit speed timelike isophote curve on a timelike surface
(The case 3(a), the case 3(b) and the case 4, respectively). Then,\newline
\textbf{(a)} $\alpha $ is a geodesic on the timelike surface if and only if $%
\alpha $ is a slant helix with the spacelike axis%
\begin{equation*}
d=\pm \frac{\tau }{\sqrt{\kappa ^{2}-\tau ^{2}}}\sin \delta T\pm \frac{%
\kappa }{\sqrt{\kappa ^{2}-\tau ^{2}}}\sin \delta B+\cos \delta N.
\end{equation*}%
\textbf{(b)} $\alpha $ is a geodesic on the timelike surface if and only if $%
\alpha $ is a slant helix with the spacelike axis%
\begin{equation*}
d=\pm \frac{\tau }{\sqrt{\tau ^{2}-\kappa ^{2}}}\sinh \xi T\pm \frac{\kappa 
}{\sqrt{\tau ^{2}-\kappa ^{2}}}\sinh \xi B+\cosh \xi N.
\end{equation*}%
\textbf{(c)} $\alpha $ is a geodesic on the timelike surface if and only if $%
\alpha $ is a slant helix with the timelike axis%
\begin{equation*}
d=\pm \frac{\tau }{\sqrt{\tau ^{2}-\kappa ^{2}}}\cosh \nu T\pm \frac{\kappa 
}{\sqrt{\tau ^{2}-\kappa ^{2}}}\cosh \nu B+\sinh \nu N.
\end{equation*}%
The proof of Theorem 6 can be done similar to Theorem 3.
\end{theorem}


\begin{thebibliography}{99}
\bibitem{ali} Ali, A.T. and Lopez, R.: Slant helices in minkowski space $%
\mathbb{E}_{1}^{3}$. J. Korean Math. Soc. 48 (1), 159-167 (2011)

\bibitem{dillen} Dillen, F., Fastenakels, J., Veken, V. J., Vrancken, L.:
Constant angle surfaces in $\mathbb{S}^{2}\mathbb{\times R}$. Monatsh. Math.
152, 89--96 (2007)

\bibitem{dillen} Dillen, F. and Munteanu, M.I.: Constant angle surfaces in $%
\mathbb{H}^{2}\mathbb{\times R}$. Bull. Braz. Math. Soc. 40 (1), 85--97
(2009)

\bibitem{doðan} Do\u{g}an, F. and Yayl\i , Y.: On isophote curve and its
characterizations (Submitted to journal)

\bibitem{doðan} Do\u{g}an, F. and Yayl\i ,\ Y.: Isophote curves on spacelike
surfaces in Lorentz-Minkowski space $E_{1}^{3}$ (Submitted to journal)

\bibitem{izumiya} Izumiya, S. and Takeuchi, N.: New special curves and
developable surfaces. Turk. J. Math. Vol. 28, 153-163 (2004)

\bibitem{uðurlu} U\u{g}urlu, H.H. and Topal, A.: Relation between Darboux
instantaneous rotation vectors of curves on a time-like surface.
Mathematical \& Computational Applications vol. 1 (2), 149-157 (1996)

\bibitem{kim} Kim, K.-J. and Lee, I.-K.: Computing isophotes of surface of
revolution and canal surface. Computer Aided Design 35, 215-223 (2003)

\bibitem{koenderink} Koenderink, J.J. and van Doorn, A.J.: Photometric
invariants related to solid shape. Journal of Modern Optics vol. 27 (7),
981-996 (1980)

\bibitem{kühnel} K\"{u}hnel, W.: Differential Geometry
Curves-Surfaces-Manifolds. Friedr. Vieweg \& Sohn Verlag, Wiesbaden (2003)

\bibitem{lopez} Lopez, R.: Differential Geometry of Curves and Surfaces in
Lorentz-Minkowski Space. arXiv:0810.3351v1

\bibitem{munteanu} Munteanu, M.I. and Nistor, A.I.: A new approach on
constant angle surfaces in $\mathbb{E}^{3}$. Turkish J. Math. 33 (2), 1--10
(2009)

\bibitem{nistor} Nistor, A.I.: Certain Constant angle surfaces constructed
on curves. International Electronic Journal of Geometry vol 4 (1), 79-87
(2011)

\bibitem{o'neill} O'neill, B.: Semi Riemannian Geometry with Applications to
Relativity. Academic Press, New York (1983)

\bibitem{o'neill} O'neill, B.: Elementary Differential Geometry. Academic
Press, New York (2006)

\bibitem{poeschl} Poeschl, T.: Detecting surface irregularities using
isophotes. Computer Aided Geometric Design 1, 163-168 (1984)

\bibitem{pressley} Pressley, A.: Elementary Differential Geometry.
Springer-Verlag, New York (2001)

\bibitem{sara} Sara, R.: Local Shading Analysis via Isophotes Properties.
Ph.D Thesis, Johannes Kepler University (1994)

\bibitem{web} http://nccastaff.bournemouth.ac.uk/jmacey/CGF/slides/\newline
IlluminationModels4up/.pdf

\bibitem{kazaz} Kazaz, M., Ugurlu, H.H., Onder, M., Kahraman, T. Mannheim
Partner D-Curves in Minkowski 3-space $\mathbb{E}_{1}^{3}$,
arXiv:1003.2043v3.

\bibitem{monterde} J. Monterde, Salkowski curves revisited: A family of
curves with constant curvature and non-constant torsion, Computer Aided
Geometric Design, 26\textbf{\ (}2009), 271-278.
\end{thebibliography}
\end{document}